\documentclass[12pt]{article}
\usepackage[utf8]{inputenc}
\usepackage{amsmath, amssymb}
\usepackage{graphicx}
\usepackage{geometry}
\usepackage{amsmath}
\usepackage{amssymb}
\usepackage{amsfonts}
\usepackage{comment}
\usepackage{graphicx}

\usepackage{float}

\usepackage{caption}

\usepackage{subcaption}

\usepackage{setspace}
\usepackage{enumitem}
\usepackage{bigints}
\usepackage{xcolor}

\usepackage{amsthm}
\usepackage[
    colorlinks=true,
    linkcolor=blue,   
    citecolor=red,    
    urlcolor=black     
]{hyperref}

\newtheorem{theorem}{Theorem}[section]
\newtheorem{lemma}[theorem]{Lemma}
\newtheorem{proposition}[theorem]{Proposition}

\theoremstyle{definition}

\theoremstyle{remark}
\newtheorem{remark}[theorem]{Remark}

\usepackage[capitalise,nameinlink]{cleveref}

\everymath{\displaystyle}

\title{Analytical  Study of Minimizers of an N-Field System with Sphere-Valued Constraints}
\author{
Khaled Chacouche$^{\dagger}$,
Rejeb Hadiji$^{\ddagger}$,
Nahla Noun-Beydoun$^{\dagger}$
}

\date{}

\begin{document}

\maketitle
\begin{center}
{\small

$^{\ddagger}$ \textbf{Rejeb Hadiji}\\
Univ Paris Est Creteil, CNRS, LAMA, F-94010 Creteil, France.
Univ Gustave Eiffel, LAMA, F-77447 Marne-la-Vallee, France.\\
E-mail: \href{mailto:rejeb.hadiji@u-pec.fr}{rejeb.hadiji@u-pec.fr}

\vspace{0.4cm}

$^{\dagger}$ \textbf{Khaled Chacouche, Nahla Noun-Beydoun}\\
LyRIDS, ECE Paris, Graduate School of Engineering, 10 rue Sextius Michel, 75015 Paris, France.\\
E-mail: \href{mailto:kchacouche@ece.fr}{kchacouche@ece.fr} ; 
\href{mailto:nbeydoun@ece.fr}{nbeydoun@ece.fr}
}
\end{center}

\vspace{0.2cm}

\begin{abstract}
This paper investigates a variational model that describes a system of $N$ coupled sphere-valued fields through a penalization term imposing the constraint that their sum coincides with a prescribed sphere value map. We analyze the asymptotic behavior and regularity of minimizers as the penalization parameter tends to infinity. A fundamental dichotomy is established according to the number of interacting fields. For $N=2$, topological obstructions may prevent the existence of exact sphere-valued decompositions, leading to the blow-up of the minimal energy in the large-penalization regime. In contrast, for $N\geq 3$, exact decompositions always exist, yielding uniform energy bounds and allowing the identification of the limiting constrained problem via $\Gamma$-convergence. We further derive an explicit characterization of the limiting energy by decomposing admissible configurations into an average field and fluctuation components. Finally, we establish a gap phenomenon showing that, under suitable topological assumptions, every minimizer is necessarily singular for sufficiently large values of the penalization parameter. These results highlight the fundamental role of topology in the asymptotic behavior and regularity of coupled sphere-valued variational systems and extend previous results on single-field models to the multi-field setting.
\end{abstract}

\medskip


\medskip


\section{Introduction}

\subsection{Sphere-valued variational problems}
Variational problems for maps taking values in the unit sphere
$S^{2}$ form an important class of constrained nonlinear problems in
geometric analysis and mathematical physics. They arise in models of
nematic liquid crystals \cite{B87,GP,MAJ,A}, micromagnetics
\cite{BG}, ferromagnetic materials, nonlinear sigma models, and more
generally, systems with manifold-valued order parameters. In such models, the pointwise spherical constraint reflects the
conservation of the magnitude of an underlying physical quantity, such
as the molecular orientation in liquid crystals or the magnetization
in ferromagnets.

The extension of the single-field framework to systems of multiple
interacting maps is motivated by their occurrence in multi-component
liquid crystals \cite{LC1, LC2, LC3}, multilayer magnetic structures, and
composite ferromagnetic materials \cite{MG1, MG2, MG3}. In these settings,
coupled order parameters describe complex collective phenomena, including
the emergence of defects, pattern formation, and phase transitions driven
by the synergy between geometric constraints and mutual field interactions.
Beyond physical models, such multi-field systems are increasingly relevant
in directional image processing, where they provide a mathematical basis
for the decomposition, denoising, and fusion of manifold-valued data \cite{IM1}.
From an analytical perspective, a fundamental challenge lies in determining
how this coupling mechanism dictates the existence, qualitative properties,
and regularity of configurations, as well as the formation of singularities
within the system \cite{MG4}.

The analysis of sphere-valued Sobolev maps has developed
substantially over the past decades. The pioneering works of Schoen
and Uhlenbeck \cite{SU1,SU2} established regularity theory to minimize harmonic maps, while Hardt, Kinderlehrer, and Lin \cite{HKL} proved the existence and partial regularity results for liquid
crystal configurations. The density of smooth sphere-valued maps and
the topological obstructions to strong approximation were studied by
Bethuel, Brezis, and Coron \cite{BBC,BCL}, and by Hardt and Lin
\cite{HL}. Together, these works established the analytical and
topological framework of modern theory of manifold-valued variational problems \cite{LM,Z}.

\subsection{The penalized single-field model}

The present work is motivated by the penalized harmonic map model
introduced in \cite{HZ}, see also \cite{HaZh}. Let
$\Omega\subset\mathbb{R}^{3}$ be a bounded domain,
$\lambda>0$ a penalization parameter, and
$f\in H^{1}(\Omega,S^{2})$ a prescribed sphere-valued map.
The authors considered the minimization problem associated with the energy
\begin{equation}\label{F_lambda_singl_field}
F_{\lambda}(u)
=
\bigintsss_{\Omega}
\left(
\frac12 |\nabla u|^{2}
+
\lambda |u-f|^{2}
\right)\,dx,
\qquad
u\in H^{1}(\Omega,S^{2}).
\end{equation}

\noindent The functional consists of the Dirichlet energy together with a
quadratic penalization term that enforces the proximity of the unknown
map to the prescribed sphere-valued map. The asymptotic analysis carried out in
\cite{HZ} and \cite{HaZh} showed that the behavior of minimizing solutions in the limit as the penalization parameter
$\lambda$ tends to infinity  is closely related to the topology of the target
manifold and to the approximation properties of
$H^{1}(\Omega,S^{2})$ by smooth sphere-valued maps. More recently,
additional regularity results together with numerical simulations
illustrating the qualitative behavior of minimizing configurations
were obtained in \cite{HaZh}. In a related direction, Gaudiello and Hadiji \cite{GH09} performed an 
asymptotic analysis of minimizing maps with values in $S^2$ in a thin multidomain 
setting, highlighting the role of geometric constraints in the limiting behavior.\\

\subsection{A multi-field model for interacting sphere-valued fields}
The primary objective of the present work is to extend the penalized single-field framework introduced in \cite{HZ, HaZh} to systems of interacting sphere-valued maps. From a mathematical perspective, coupling several constrained fields introduces new variational phenomena that are entirely absent in the classical single-field setting.\\

\noindent For an integer $N\ge2$, we define the multi-field configuration space
\begin{equation}\label{fields_space}
\mathcal{H}:=H^{1}(\Omega,S^{2})^{N}.
\end{equation}
For
$U=(u_{1},\ldots,u_{N})\in\mathcal{H}$,
we consider the penalized energy functional
\begin{equation}\label{Functional}
F_{\lambda}(U)
=
\bigintss_{\Omega}
\left(
\frac12\sum_{i=1}^{N}|\nabla u_i|^{2}
+
\lambda
\left|
\sum_{i=1}^{N}u_i-f
\right|^{2}
\right)\,dx.
\end{equation}

\noindent The first term corresponds to the total Dirichlet energy of the interacting
system, while the second introduces a global coupling through 
penalization of the constraint
\begin{equation}\label{constraint_intro}
\sum_{i=1}^{N}u_i=f.
\end{equation}

\noindent The associated minimization problem is given by
\begin{equation}\label{MinProblem}
E_{\lambda}(f)
:=
\inf_{U\in\mathcal{H}}
F_{\lambda}(U).
\end{equation}

By the direct method of the calculus of variations, the functional
$F_{\lambda}$ admits at least one minimizer
\[
U^\lambda=(u_1^\lambda,\ldots,u_N^\lambda)\in\mathcal{H},
\]
which satisfies
\[
E_\lambda(f)
= \min_{U\in\mathcal{H}}
F_{\lambda}(U)
=F_\lambda(U^\lambda).
\]

\noindent Moreover, every minimizer satisfies the coupled Euler--Lagrange system:
\begin{equation}
-\Delta u_i^\lambda
=
u_i^\lambda|\nabla u_i^\lambda|^{2}
+
2\lambda
\left(
f-\sum_{j=1}^{N}u_j^\lambda
-
\left\langle
u_i^\lambda,
f-\sum_{j=1}^{N}u_j^\lambda
\right\rangle
u_i^\lambda
\right),
\qquad i=1,\ldots,N.
\end{equation}

\noindent The coupling term induces strong interactions among the different
components, substantially enriching the mathematical structure
of the problem. In contrast to the single-field model, the interplay
between the geometric constraint and the coupling mechanism gives rise
to new analytical features whose analysis requires refined variational
techniques.

\noindent As in the classical theory of minimizing harmonic maps
\cite{SU1,SU2,LM}, minimizing configurations may develop
singularities. One of the main objectives of the present work is to
investigate how the coupling mechanism affects the asymptotic behavior
of minimizers, the associated limiting variational problem, and the
formation of singular configurations.

\subsection{Asymptotic regime and exact decompositions}

We investigate the asymptotic behavior of the minimization problem in the
strong penalization regime, as $\lambda$ tends to $+\infty$. In this limit, the penalization term
asymptotically enforces the coupling constraint
\begin{equation}\label{constraint_intro2}
\sum_{i=1}^{N}u_i=f,
\end{equation}
so that the limiting behavior of minimizers is entirely determined by
the existence of admissible sphere-valued decompositions of the
prescribed map $f$. Consequently, the asymptotic properties of both
the minimizing configurations and the associated minimal energies are
closely related to the solvability of the above nonlinear constraint.

A fundamental distinction arises depending on the number of interacting fields. For $N=2$, the coupling constraint may not admit admissible
solutions in $H^{1}(\Omega,S^{2})^{2}$ due to topological
obstructions. More precisely, a decomposition of the form
\[
f=u_{1}+u_{2},
\qquad
u_{1},u_{2}\in H^{1}(\Omega,S^{2}),
\]
does not necessarily exist. This lack of admissible configurations has
a direct impact on the variational problem, leading to the blow-up of
the minimal energy as the penalization parameter tends to infinity.

The situation is fundamentally different when $N\ge3$. In this case,
every map with sphere values admits an exact decomposition satisfying the
coupling constraint. This structural property guarantees the existence
of uniformly bounded admissible competitors, yields uniform estimates
for the minimal energy, and provides the appropriate compactness
framework required for the asymptotic analysis of the family
$\left(F_{\lambda}\right)_{\lambda>0}$.
It furthermore allows the identification of the limiting constrained
variational problem together with an explicit characterization of the
corresponding limiting energy.

\subsection{Gap phenomenon and singular minimizers}

A further objective of this work is to investigate the relationship between topological obstructions and singular minimizers. To this end, we consider the restricted minimization problem
\begin{equation}\label{RegularEnergy}
E_{\lambda,\mathrm{reg}}(f)
=
\inf_{U\in C^{1}(\overline{\Omega},S^{2})^{N}}
F_{\lambda}(U),
\end{equation}
where the admissible class is restricted to smooth
sphere-valued configurations.

\noindent This notion is closely related to the gap phenomena established by
Bethuel, Brezis and Coron \cite{BBC,BCL}, and later investigated by
Hadiji and Zhou \cite{HZ}.
Under a suitable topological obstruction that prevents the decomposition
of $f$ into strongly approximable sphere-valued maps, we
prove the existence of a constant $\lambda_{0}>0$ such that
\begin{equation}\label{hcn_singularity}
E_{\lambda}(f)
<
E_{\lambda,\mathrm{reg}}(f),
\qquad
\forall\,\lambda\ge\lambda_{0}.
\end{equation}

\noindent The strict inequality shows that the infimum over smooth admissible
configurations cannot be attained by any minimizer. As a
consequence, every minimizer of $F_{\lambda}$ necessarily develops at
least one singular point for sufficiently large values of the
penalization parameter. 

\subsection{Main contributions and organization of the paper}

The first main result establishes a fundamental dichotomy between the
cases $N=2$ and $N\ge3$. It shows that topological obstructions may
prevent the existence of exact sphere-valued decompositions when
$N=2$, leading to the divergence of the minimal energy, whereas for
$N\ge3$ exact decompositions always exist. This property yields uniform
energy bounds, allows the identification of the limiting constrained
problem through $\Gamma$-convergence, and provides an explicit
characterization of the limiting energy.

The second main result investigates the interplay between topology and
regularity. Under suitable topological assumptions, we establish a gap
phenomenon between unrestricted and regular minimizers, implying that,
for sufficiently large values of the penalization parameter, every
minimizer necessarily develops singularities.

The remainder of the paper is organized as follows.
Section~\ref{sec:mainresults} states the main results.
Section~\ref{sec:proofs} is devoted to the analysis of case 
$N=2$, where topological obstructions lead to the divergence of the
minimal energy.
Section~\ref{sec:asymptotic} treats the case $N\ge3$, establishes the
$\Gamma$-convergence of penalized functionals, derives an explicit
representation of the limiting energy, and concludes with the proof of
the gap phenomenon and the occurrence of singular minimizers.

\section{Main Results}
\label{sec:mainresults}
Let  
\begin{equation}\label{Hf}
\mathcal H_f :=
\left\{
U=(u_1,\ldots,u_N)\in\mathcal H:
\sum_{i=1}^{N}u_i=f
\ \text{a.e. in }\Omega
\right\},
\end{equation}
and 

\begin{equation}\label{Wf}
W_f :=
\left\{
(w_1,\ldots,w_N)\in
\bigl(H^1(\Omega,\mathbb R^3)\bigr)^N:
\sum_{i=1}^N w_i=0,\;
\left|\frac1Nf+w_i\right|=1
\ \text{a.e. in }\Omega
\right\}.
\end{equation}

For  \(N\ge3\), we introduce the limiting energy:
\begin{equation}\label{hcn_energy_infinity}
E_\infty(f)
=
\inf_{U\in\mathcal H_f}F_{\infty}(U)
\end{equation}
where \begin{equation}\label{Finfinity}
F_{\infty}:\mathcal H_{f}\to\mathbb R,
\qquad
F_{\infty}(U)
=
\frac12
\sum_{i=1}^{N}
\int_{\Omega}|\nabla u_{i}|^{2}\,dx.
\end{equation}

We are now in a position to state the main results of this paper. The first theorem concerns the case $N=2$ and describes the asymptotic behavior of the minimal energy $E_{\lambda}(f)$ for representative prescribed maps. The second theorem addresses the case $N\geq 3$ and establishes the existence of exact decompositions, the $\Gamma$-convergence of the penalized functionals, an explicit characterization of the limiting energy, and the occurrence of singular minimizers under suitable topological assumptions.

\begin{theorem}
\label{thm:asymptotic}

Assume that $N=2$ and let $f\in H^1(\Omega, S^2) $.

\begin{enumerate}

\item If $f(x)=\frac{x}{|x|},$ then the following assertions hold:
\begin{enumerate}

\item [(a)] The minimal energy satisfies  \begin{equation}\label{N2EnergyDivergence}
\lim_{\lambda\to+\infty}E_{\lambda}(f)=+\infty.
\end{equation}

\item  [(b)] There exists a sequence $\lambda_n$ that tends to $+\infty$ such that, for every $n\in\mathbb N$,
\begin{equation}\label{noneq}
u_{1,n}+u_{2,n}\neq f.
\end{equation}
\end{enumerate}

\item
If $f = C\in S^{2}$ is constant, then
\[
E_{\lambda}(f)=0,
\qquad
\forall \lambda>0 .
\]

\end{enumerate}

\end{theorem}

\begin{theorem}\label{thm_singularity}

 Assume that $N\ge3$ and let $ f\in H^1(\Omega,S^2).$

\begin{enumerate}

\item
The admissible class $\mathcal H_f$ defined in \eqref{Hf} is nonempty.

\item
The family $(F_{\lambda})_{\lambda>0}$
$\Gamma$-converges, as $\lambda$ tends to  $+\infty$, with respect to the weak topology of
$\mathcal H$, to the functional $F_{\infty}$.
Consequently,
\begin{equation}
\lim_{\lambda\to+\infty}E_{\lambda}(f)=E_{\infty}(f).
\end{equation}

\item
Furthermore,

\begin{equation}
E_\infty(f)
=
\frac1{2N}
\int_\Omega |\nabla f|^2\,dx
+
\min_{(w_1,\ldots,w_N)\in W_f}
\frac12
\sum_{i=1}^N
\int_\Omega |\nabla w_i|^2\,dx.
\end{equation}

\item Assume that $f$ cannot be decomposed as $f=\sum_{i=1}^N u_i$ where each map $u_i\in H^1(\Omega,S^2)$ is a strong limit of smooth maps in $H^1(\Omega,S^2)$. Then there exists $\lambda_0>0$, depending on $f$ and $N$, such that for all $\lambda \ge \lambda_0$, any minimizer $(u_1^\lambda,\dots,u_N^\lambda)$ of $F_\lambda$ is not regular in $\Omega$. In particular, at least one
component \(u_i^\lambda\) admits a singularity in \(\Omega\). Moreover, 
\begin{equation}\label{GapLargeLambda}
E_{\lambda}(f)<E_{\lambda,\mathrm{reg}}(f).
\end{equation}
\end{enumerate}

\end{theorem}

\section{Proof of Theorem \ref{thm:asymptotic} (Case \(N=2\))}
\label{sec:proofs}


For \(N=2\), the behavior of the minimization problem differs
substantially from that of the regime \(N\ge3\). Indeed, the
constraint
\[
u_1+u_2=f
\]
 cannot, in general, be realized by sphere-valued Sobolev maps. Topological obstructions may prevent the existence of exact decompositions, leading to the divergence of the minimal energy, as $\lambda$ tends to $+\infty$. On the other hand, for sufficiently regular
data, such as constant maps, the constraint can be satisfied exactly, and the penalization term vanishes identically.\\

The first result establishes the existence of a topological
obstruction for the radial map
\[
f(x)=\frac{x}{|x|},
\]
which implies the blow-up of the minimal energy in the asymptotic
regime when $\lambda$ tends to  $+\infty$. We then show that this obstruction
disappears for constant maps, for which the minimum energy is identically equal to zero. \\


The following propositions make these statements precise.

\begin{proposition}\label{N=2noexistence}
Let $\Omega\subset \mathbb{R}^3$ be an open set that contains the origin. Then there do not exist maps $u_1,u_2\in H^1(\Omega,S^2)$ such that
\[
u_1+u_2=\frac{x}{|x|}
\qquad \text{a.e. in }\Omega.
\]
\end{proposition}

\begin{proof}
Set
\[
f(x)=\frac{x}{|x|}.
\]

Assume by contradiction that there exist
$ u_1,u_2\in H^1(\Omega,S^2)$
satisfying
\[
u_1+u_2=f.
\]

Since $|u_1|=|u_2|=|f|=1$, we have
\[
1=|u_1+u_2|^2
=2+2u_1\cdot u_2,
\]
and therefore
\[
u_1\cdot u_2=-\frac12.
\]

Define
\[
m:=\frac{u_1-u_2}{\sqrt3}.
\]

Then, we have
\[
u_1=\frac12 f+\frac{\sqrt3}{2}m,
\qquad
u_2=\frac12 f-\frac{\sqrt3}{2}m,
\]

with

\[
|m|=1 \quad \mbox{and} \quad 
m(x)\cdot  f(x)=m(x)\cdot \frac{x}{|x|}=0 \quad\text{a.e. in }\Omega.
 \]

Since $u_1,u_2\in H^1(\Omega,\mathbb R^3)$, it follows that
$m\in H^1(\Omega,\mathbb R^3).$

Let $r>0$ be such that $\overline{B_r}\subset\Omega$. By the trace theorem (or the coarea formula), for almost every such $r$,
\[
m_{|S_r}\in H^1(S_r,\mathbb R^3),
\quad
\mbox{where} \quad
S_r=\partial B_r.
\]

\noindent Furthermore,
\[
m(x)\cdot x=0,
\qquad
|m(x)|=1 \quad\text{ a.e. in } S_r .\]

\noindent Therefore, $m_{|S_r}$ is an $H^1$ unit of the tangent vector field in the sphere $S_r$. After rescaling $S_r$ to the unit sphere $S^2$, we obtain a map
\[
\widetilde m\in H^1(S^2,\mathbb R^3)
\]
such that
\[
\widetilde m(x)\cdot x=0,
\qquad
|\widetilde m(x)|=1 \quad\text{ a.e. in } S^2.\]

\noindent In other words, $\widetilde m$ is an $H^1$ unit section of the tangent bundle $TS^2$. However, by the hairy ball theorem (equivalently, by the non-triviality of the tangent bundle of $S^2$), there exists no $H^1$ unit tangent vector field on $S^2$.\\
\noindent This contradiction proves that such maps $u_1$ and $u_2$ cannot exist.
\end{proof}

\bigskip

The next two lemmas are valid for every $N\ge2$. They establish basic structural properties of the minimal energy that will be used repeatedly throughout the rest of the paper. To this end, we introduce the following notation.
For every $U=(u_1,\ldots,u_N)\in \mathcal H$, set
\[
A(U):=\frac12\sum_{i=1}^N\int_\Omega |\nabla u_i|^2\,dx,
\qquad
B(U):=\int_\Omega\left|\sum_{i=1}^N u_i-f\right|^2\,dx.
\]
Then
\[
F_\lambda(U)=A(U)+\lambda B(U),
\qquad
E_\lambda(f)=\min_{U\in \mathcal H}F_\lambda(U).
\]

\begin{lemma}\label{properties}
The map
\[
\lambda\longmapsto E_\lambda(f)
\]
is Lipschitz continuous, nondecreasing, and concave on $(0,+\infty)$.
\end{lemma}

\begin{proof}
Since $|u_i|=1$ and $|f|=1$ almost everywhere in $\Omega$, we have, for every $U\in \mathcal H$,
\[
\left|\sum_{i=1}^N u_i-f\right|
\le \sum_{i=1}^N |u_i|+|f|
\le N+1.
\]
Therefore,
\[
0\le B(U)\le |\Omega|(N+1)^2=:C.
\]

Let $\lambda,\mu>0$, and let $U_\mu\in\mathcal H$ be a minimizer of $F_\mu$. Then
\[
E_\lambda(f)
\le F_\lambda(U_\mu)
=F_\mu(U_\mu)+(\lambda-\mu)B(U_\mu),
\]
and hence
\[
E_\lambda(f)-E_\mu(f)
\le C|\lambda-\mu|.
\]
Exchanging the roles of $\lambda$ and $\mu$ yields
\[
|E_\lambda(f)-E_\mu(f)|
\le C|\lambda-\mu|,
\]
which proves the Lipschitz continuity.

Now, let $0<\lambda_1<\lambda_2$, and let
$U_{\lambda_2}\in\mathcal H$ be a minimizer of $F_{\lambda_2}$. Since
$B(U_{\lambda_2})\ge0$,
\[
E_{\lambda_1}(f)
\le F_{\lambda_1}(U_{\lambda_2})
=F_{\lambda_2}(U_{\lambda_2})
-(\lambda_2-\lambda_1)B(U_{\lambda_2})
\le E_{\lambda_2}(f).
\]
Thus, $\lambda\mapsto E_\lambda(f)$ is not decreasing.

Finally, let $\lambda_1,\lambda_2>0$, $t\in[0,1]$, and set
\[
\lambda_t:=t\lambda_1+(1-t)\lambda_2.
\]
For every $U\in\mathcal H$,
\[
F_{\lambda_t}(U)
=tF_{\lambda_1}(U)+(1-t)F_{\lambda_2}(U).
\]
Taking the minimum over $U\in\mathcal H$, we obtain
\[
\begin{aligned}
E_{\lambda_t}(f)
&=\min_{U\in\mathcal H}
\left(
tF_{\lambda_1}(U)+(1-t)F_{\lambda_2}(U)
\right)\\
&\ge
t\min_{U\in\mathcal H}F_{\lambda_1}(U)
+(1-t)\min_{U\in\mathcal H}F_{\lambda_2}(U)\\
&=
tE_{\lambda_1}(f)+(1-t)E_{\lambda_2}(f).
\end{aligned}
\]
Hence $\lambda\mapsto E_\lambda(f)$ is concave.
\end{proof}

\begin{lemma}
For each $\lambda>0$, let $U^\lambda=(u_1^\lambda,\ldots,u_N^\lambda)\in\mathcal H$ be a minimizer of $F_\lambda$. Then the map 
$\lambda\longmapsto E_\lambda(f)$ 
is differentiable almost everywhere on $(0,+\infty)$ and satisfies
\begin{equation}\label{derivative}
\frac{d}{d\lambda}E_\lambda(f)
=
\int_\Omega
\left|
\sum_{i=1}^N u_i^\lambda-f
\right|^2\,dx
\qquad\text{for almost every }\lambda>0.
\end{equation}
\end{lemma}

\begin{proof}
Let
\[
R_\lambda
=
\int_\Omega
\left|
\sum_{i=1}^N u_i^\lambda-f
\right|^2dx.
\]

\noindent From Lemma ~\ref{properties}, for every $h>0$, we have
\begin{equation}\label{ineq_hcn}
\frac{E_{\lambda+h}(f)-E_\lambda(f)}{h}
\le
R_\lambda.
\end{equation}
Since $E_\lambda(f)$ is continuous and strictly increasing, it is differentiable a.e. $\lambda>0.$
Therefore, passing to the limit as $h$ tends to $0^+$ gives
\begin{equation}\label{droit}
R_\lambda
\ge
\lim_{h\to0^+}
\frac{E_{\lambda+h}(f)-E_\lambda(f)}{h}
=
\frac{d^+}{d\lambda}E_\lambda(f).   
\end{equation}

\noindent Similarly, applying inequality \eqref{ineq_hcn} with $\lambda-h$ instead of $\lambda$,
\[
R_\lambda
\le
\frac{E_\lambda(f)-E_{\lambda-h}(f)}{h}.
\]

\noindent Passing to the limit as $h$ tends to $0^+$ yields
\begin{equation}\label{gauche}
R_\lambda
\le
\lim_{h\to0^+}
\frac{E_\lambda(f)-E_{\lambda-h}(f)}{h}
=
\frac{d^-}{d\lambda}E_\lambda(f).
\end{equation}

\noindent Combining \eqref{droit} and \eqref{gauche}, and using the fact that $E_\lambda(f)$ is differentiable almost everywhere in $(0,+\infty)$, yields \eqref{derivative}. This completes the proof.
\end{proof}

\medskip

The following proposition establishes the first two assertions of Theorem \ref{thm:asymptotic} for the case $N=2$.

\begin{proposition}\label{prop:N2Divergence}
Let \(\Omega\subset\mathbb R^{3}\) be an open set that contains the origin. Assume that
\[
N=2,
\qquad
f(x)=\frac{x}{|x|}.
\]
Then the following assertions hold:
\begin{enumerate}

\item The minimal energy satisfies  \begin{equation}\label{N2EnergyDivergence}
\lim_{\lambda\to+\infty}E_{\lambda}(f)=+\infty.
\end{equation}

\item  There exists a sequence $\lambda_n$ that tends to $+\infty$ such that, for every $n\in\mathbb N$,
\begin{equation}\label{noneq}
u_{1,n} +u_{2,n} \neq f.
\end{equation}
\end{enumerate}

\end{proposition}

\begin{proof}
We first prove \eqref{N2EnergyDivergence}.\\
Assume by contradiction that there exist \(\lambda_{n}\) tending to $+\infty $ and maps $u_{1,n},u_{2,n}\in H^{1}(\Omega,S^{2})$ such that
\begin{equation}\label{BoundedEnergyN2Contradiction}\nonumber
F_{\lambda_{n}}(u_{1,n},u_{2,n})\le C
\end{equation}
for some constant \(C>0\) independent of \(n\). Then
\begin{equation}\label{DirichletBoundsN2}\nonumber
\int_{\Omega}|\nabla u_{1,n}|^{2}\,dx\le C,
\qquad
\int_{\Omega}|\nabla u_{2,n}|^{2}\,dx\le C,
\end{equation}
and
\begin{equation}\label{PenaltyBoundN2}\nonumber
\int_{\Omega}
|u_{1,n}+u_{2,n}-f|^{2}\,dx
\le
\frac{C}{\lambda_{n}}.
\end{equation}
Hence
\begin{equation}\label{L2convergence}
u_{1,n}+u_{2,n}\to f
\quad\text{strongly in }L^{2}(\Omega).
\end{equation}
Moreover, \((u_{1,n})\) and \((u_{2,n})\) are bounded in \(H^{1}(\Omega,S^{2})\). Therefore, up to a subsequence, there exist
\[
u_{1},u_{2}\in H^{1}(\Omega,S^{2})
\]
such that
\begin{equation}\label{WeakConvN2}\nonumber
u_{1,n}\rightharpoonup u_{1},
\qquad
u_{2,n}\rightharpoonup u_{2}
\quad\text{weakly in }H^{1}(\Omega,S^{2}),
\end{equation}
and
\begin{equation}\label{StrongConvN2}
u_{1,n}\to u_{1},
\qquad
u_{2,n}\to u_{2}
\quad\text{strongly in }L^{2}(\Omega).
\end{equation}
By \eqref{L2convergence}, \eqref{StrongConvN2} and uniqueness of the limit in \(L^{2}(\Omega)\),
\begin{equation}\label{LimitConstraintN2}\nonumber
u_{1}+u_{2}=f
\quad\text{a.e. in }\Omega.
\end{equation}
This contradicts Proposition~\ref{N=2noexistence}.\\

\noindent We now prove the second assertion.\\
We will show that there exist $\alpha>1$, a constant $M>0$, and a sequence $\lambda_n\longrightarrow+\infty$ such that
\begin{equation}\label{eq:lower-bound-penalty-second}\nonumber
\lambda_n^\alpha
\int_\Omega
\left|
u_{\lambda_n}^1+u_{\lambda_n}^2-f
\right|^2\,dx
\ge M,
\qquad \forall n\in\mathbb N.
\end{equation}

\noindent Otherwise, for every $M>0$, for every $\lambda>0$ and for every
$\alpha>1$, one has
\begin{equation}\label{eq:upper-bound-penalty-second}
\lambda^\alpha
\int_\Omega
\left|
u_\lambda^1+u_\lambda^2-f
\right|^2\,dx
< M.
\end{equation}

\noindent Since, for almost every $\lambda>0$,
\begin{equation}\label{eq:derivative-energy-second}\nonumber
E'(\lambda)
=
\int_\Omega
\left|
u_\lambda^1+u_\lambda^2-f
\right|^2\,dx,
\end{equation}

\noindent Using
\begin{equation}\label{eq:energy-integral-second}\nonumber
E(\lambda)
=
\int_1^\lambda E'(r)\,dr
+
E(1),
\end{equation}

\noindent from \eqref{eq:upper-bound-penalty-second}, we obtain
\begin{equation}\label{eq:energy-integral-estimate-second}\nonumber
E(\lambda)
<
\int_1^\lambda
\frac{M}{r^\alpha}\,dr
+
E(1).
\end{equation}

\noindent Consequently,
\begin{equation}\label{eq:energy-power-estimate-second}\nonumber
E(\lambda)
<
M
\frac{\lambda^{1-\alpha}-1}{1-\alpha}
+
E(1).
\end{equation}

\noindent Since $1-\alpha<0,$ we have
\begin{equation}\label{eq:uniform-bound-energy-second}\nonumber
M
\frac{1-\lambda^{1-\alpha}}{\alpha-1}
+
E(1)
\leq C.
\end{equation}

\noindent This contradicts
\begin{equation}\label{eq:energy-divergence-second}\nonumber
\lim_{\lambda\to+\infty}E(\lambda)=+\infty.
\end{equation}

\noindent Therefore, there exist $\alpha>1$, a constant $M>0$, and a sequence $\lambda_n\longrightarrow+\infty$ 
such that, for every $n\in\mathbb N$,
\begin{equation}\label{eq:positive-penalty-second}\nonumber
\int_\Omega
\left|
u_{\lambda_n}^1+u_{\lambda_n}^2-f
\right|^2\,dx
\geq
\frac{M}{\lambda_n^\alpha}
>
0.
\end{equation}

\noindent Consequently, \eqref{noneq} holds.

\end{proof}

The final assertion of Theorem \ref{thm:asymptotic} is established by the following proposition.

\begin{proposition}\label{prop:N2ConstantDatum}
Let \(\Omega\subset\mathbb R^{3}\) be a bounded domain and let
\[
f(x)=C,
\qquad
|C|=1.
\]
Then there exist maps
\[
u_{1},u_{2}\in H^{1}(\Omega,S^{2})
\]
such that
$$u_1+u_2=f \quad \mbox{ and }\quad E_\lambda(f)=0.$$
\end{proposition}

\begin{proof}
Since \(|C|=1\), we may assume, up to a rotation, that
\begin{equation}\label{CChoice}\nonumber
C=(1,0,0).
\end{equation}
Consider
\begin{equation}\label{C1C2Definition}\nonumber
C_{1}
=
\left(
\frac12,
\frac{\sqrt3}{2},
0
\right),
\qquad
C_{2}
=
\left(
\frac12,
-\frac{\sqrt3}{2},
0
\right).
\end{equation}
Then
\begin{equation}\label{C1C2Norms}\nonumber
|C_{1}|=|C_{2}|=1 \quad \mbox{and} \quad  C_{1}+C_{2}=C.
\end{equation}
Define
\begin{equation}\label{ConstantMapsN2}\nonumber
u_{1}(x):=C_{1},
\qquad
u_{2}(x):=C_{2}.
\end{equation}
Then
\[
u_{1},u_{2}\in H^{1}(\Omega,S^{2}),
\qquad
\nabla u_{1}=\nabla u_{2}=0.
\]
Moreover,
\begin{equation}\label{ConstantMapsConstraint}\nonumber
u_{1}+u_{2}=C=f.
\end{equation}
Consequently,
\begin{equation}\label{ZeroEnergyConstantN2}\nonumber
F_{\lambda}(u_{1},u_{2})=0.
\end{equation}
Thus \(E_{\lambda}(f)=0\).
\end{proof}

\section{Proof of Theorem \ref{thm_singularity} (Case \(N \geq 3\))}\label{sec:asymptotic}

We now turn to the proof of Theorem \ref{thm_singularity}, which is established in several steps.\\

The first assertion follows from the next proposition, which also provides a uniform upper bound on the minimal energy.

\begin{proposition}
\label{lem:uniform-bound} Assume that \(N\ge3\), and let \[ f\in H^{1}(\Omega,S^{2}). \] Then there exists \[ (u_{1},\ldots,u_{N})\in H^{1}(\Omega,S^{2})^{N} \] such that \begin{equation}\label{ConstraintExact} \sum_{i=1}^{N}u_{i}=f \qquad\text{a.e. in }\Omega. \end{equation} Consequently, there exists a constant \(C>0\), independent of \(\lambda\), such that \begin{equation}\label{UniformBoundElambda} E_{\lambda}(f)\le C, \qquad \forall\,\lambda>0. \end{equation} \end{proposition}

\begin{proof}
Set
\begin{equation}\label{u1f}
u_{1}=f.
\end{equation}
Since \(N\ge3\), choose constant vectors
$p_{1},\ldots,p_{N-1}\in S^{2}$
such that
\begin{equation}\label{SumPkZero}
\sum_{k=1}^{N-1}p_{k}=0.
\end{equation}
For instance, take
\begin{equation}\label{PkDefinition}\nonumber
p_{k}
=
\left(
\cos\frac{2\pi(k-1)}{N-1},
\sin\frac{2\pi(k-1)}{N-1},
0
\right),
\qquad
k=1,\ldots,N-1.
\end{equation}
Define
\begin{equation}\label{UkDefinition}
u_{k+1}=p_{k},
\qquad
k=1,\ldots,N-1.
\end{equation}
Then, by \eqref{u1f}, \eqref{SumPkZero} and \eqref{UkDefinition},
\begin{equation}\label{ExactSumProof}\nonumber
\sum_{i=1}^{N}u_{i}
=
f+\sum_{k=1}^{N-1}p_{k}
=
f.
\end{equation}
Thus
\[
(u_{1},\ldots,u_{N})\in H^{1}(\Omega,S^{2})^{N}.
\]
Moreover, all fields except \(u_{1}=f\) are constant. Hence
\begin{equation}\label{EnergyExactConfiguration}\nonumber
F_{\lambda}(u_{1},\ldots,u_{N})
=
\frac12
\int_{\Omega}|\nabla f|^{2}\,dx.
\end{equation}
Therefore,
\begin{equation}\label{UniformBoundProof}\nonumber
E_{\lambda}(f)
\le
\frac12
\int_{\Omega}|\nabla f|^{2}\,dx.
\end{equation}
Thus \eqref{UniformBoundElambda} holds with
\begin{equation}\label{ConstantUniformBound}\nonumber
C=
\frac12
\int_{\Omega}|\nabla f|^{2}\,dx.
\end{equation}
\end{proof}

\begin{lemma}\label{lem:constraint-limit}
Assume that \(N\ge3\). Let \(f\in H^{1}(\Omega,S^{2})\), and let
\((\lambda_{n})\) be a sequence of positive numbers such that $\lambda_{n}$ tends to  $+\infty.$ 
Suppose that $(u_{1}^{n},\ldots,u_{N}^{n})\in H^{1}(\Omega,S^{2})^{N}$ 
satisfies
\begin{equation}\label{UniformEnergySequence}
F_{\lambda_{n}}(u_{1}^{n},\ldots,u_{N}^{n})\le C,
\end{equation}
for some constant \(C>0\) independent of \(n\). Then, up to the extraction of a subsequence, there exist maps
\[
\tilde{u}_{1},\ldots,\tilde{u}_{N}\in H^{1}(\Omega,S^{2})
\]
such that, for every \(i=1,\ldots,N\),
\begin{equation}\label{WeakConvergenceUi}\nonumber
u_{i}^{n}\rightharpoonup \tilde{u}_{i}
\quad\text{weakly in }H^{1}(\Omega),
\end{equation}
\begin{equation}\label{StrongConvergenceUi}\nonumber
u_{i}^{n}\to \tilde{u}_{i}
\quad\text{strongly in }L^{2}(\Omega),
\end{equation}
and
\begin{equation}\label{LimitConstraint}
\sum_{i=1}^{N}\tilde{u}_{i}=f
\quad\text{a.e. in }\Omega.
\end{equation}
\end{lemma}

\begin{proof}
Assume that \eqref{UniformEnergySequence} holds. Then
\begin{equation}\label{EnergyBoundExpanded}
\frac12
\sum_{i=1}^{N}
\int_{\Omega}|\nabla u_{i}^{n}|^{2}\,dx
+
\lambda_{n}
\int_{\Omega}
\left|
\sum_{i=1}^{N}u_{i}^{n}-f
\right|^{2}\,dx
\le C.
\end{equation}
In particular, for every \(i=1,\ldots,N\), the sequence \((u_{i}^{n})\) is bounded in \(H^{1}(\Omega)\). Hence, up to a subsequence,
\begin{equation}\label{WeakLimitUi}\nonumber
u_{i}^{n}\rightharpoonup \tilde{u}_{i}
\quad\text{weakly in }H^{1}(\Omega),
\end{equation}
and 
\begin{equation}\label{StrongLimitUi}
u_{i}^{n}\to \tilde{u}_{i}
\quad\text{strongly in }L^{2}(\Omega).
\end{equation}
Since $|u_{i}^{n}|=1\quad\text{a.e. in }\Omega,$
\eqref{StrongLimitUi} implies
$
|\tilde{u}_{i}|=1
\quad\text{a.e. in }\Omega.
$ 
Thus
\begin{equation}\label{Ui0Sphere}\nonumber
\tilde{u}_{i}\in H^{1}(\Omega,S^{2}).
\end{equation}
Moreover, from \eqref{EnergyBoundExpanded},
\begin{equation}\label{PenaltyBound}\nonumber
\lambda_{n}
\int_{\Omega}
\left|
\sum_{i=1}^{N}u_{i}^{n}-f
\right|^{2}\,dx
\le C.
\end{equation}
Hence
\begin{equation}\label{PenaltyToZero}\nonumber
\int_{\Omega}
\left|
\sum_{i=1}^{N}u_{i}^{n}-f
\right|^{2}\,dx
\le
\frac{C}{\lambda_{n}}
\longrightarrow 0.
\end{equation}
Therefore,
\begin{equation}\label{SumStrongLimit}
\sum_{i=1}^{N}u_{i}^{n}
\to f
\quad\text{strongly in }L^{2}(\Omega).
\end{equation}
By \eqref{StrongLimitUi}, \eqref{SumStrongLimit} and uniqueness of the limit in \(L^{2}(\Omega)\), we obtain \eqref{LimitConstraint}.
\end{proof}
\medskip

The following proposition establishes the second assertion of Theorem \ref{thm_singularity}.

\begin{proposition}\label{thm:gamma-convergence}
Assume that \(N\ge3\). Let \( (\lambda_{n}) \) be a sequence of positive numbers such that
$\lambda_{n}$ tends to $+\infty$. Then the family \((F_{\lambda_{n}})\) \(\Gamma\)-converges, with respect to the weak topology of \(\mathcal H\), to the functional $F_{\infty}$ defined in \eqref{Finfinity}. More precisely:

\begin{enumerate}
\item[(i)] Let
$
(u_{1}^{n},\ldots,u_{N}^{n})\in\mathcal H
$
be such that $(u_{1}^{n},\ldots,u_{N}^{n})
\rightharpoonup
(\tilde{u}_{1},\ldots,\tilde{u}_{N})
\quad\text{in }\mathcal H.$ Then
\begin{equation}\label{LimitInHfGamma}\nonumber
(\tilde{u}_{1},\ldots,\tilde{u}_{N})\in\mathcal H_{f}\quad \mbox{and} \quad F_{\infty}(\tilde{u}_{1},\ldots,\tilde{u}_{N})
\le
\liminf_{n\to\infty}
F_{\lambda_{n}}(u_{1}^{n},\ldots,u_{N}^{n}).
\end{equation}

\item[(ii)] For every $(\tilde{u}_{1},\ldots,\tilde{u}_{N})\in\mathcal H_{f},$
there exists a sequence $(u_{1}^{n},\ldots,u_{N}^{n})\in\mathcal H$
such that
\begin{equation}\label{RecoveryWeakConv}\nonumber
(u_{1}^{n},\ldots,u_{N}^{n})
\rightharpoonup
(\tilde{u}_{1},\ldots,\tilde{u}_{N})
\quad\text{in }\mathcal H,
\end{equation}
and
\begin{equation}\label{GammaLimsup}\nonumber
\limsup_{n\to\infty}
F_{\lambda_{n}}(u_{1}^{n},\ldots,u_{N}^{n})
\le
F_{\infty}(\tilde{u}_{1},\ldots,\tilde{u}_{N}).
\end{equation}
\end{enumerate}

Consequently,
\begin{equation}\label{MinimaConvergenceGamma}
E_{\lambda_{n}}\longrightarrow E_{\infty}.
\end{equation}
\end{proposition}

\begin{proof}
We first prove the lower bound. Let
\[
(u_{1}^{n},\ldots,u_{N}^{n})
\rightharpoonup
(\tilde{u}_{1},\ldots,\tilde{u}_{N})
\quad\text{weakly in }\mathcal H.
\]
If $\liminf_{n\to\infty}
F_{\lambda_{n}}(u_{1}^{n},\ldots,u_{N}^{n})
=
+\infty,$ there is nothing to prove. Otherwise, up to a subsequence, there exists \(C>0\) such that
\begin{equation}\label{BoundedEnergyGamma}\nonumber
F_{\lambda_{n}}(u_{1}^{n},\ldots,u_{N}^{n})\le C
\qquad
\forall n.
\end{equation}
Since \(N\ge3\), Lemma~\ref{lem:constraint-limit} gives
\begin{equation}\label{HfLimitGammaProof}\nonumber
(\tilde{u}_{1},\ldots,\tilde{u}_{N})\in\mathcal H_{f}.
\end{equation}
By weak lower semicontinuity,
\begin{equation}\label{DirichletLscGamma}\nonumber
\frac12
\int_{\Omega}
\sum_{i=1}^{N}|\nabla \tilde{u}_{i}|^{2}\,dx
\le
\liminf_{n\to\infty}
\frac12
\int_{\Omega}
\sum_{i=1}^{N}|\nabla u_{i}^{n}|^{2}\,dx.
\end{equation}
Since the penalization term is nonnegative,
\begin{equation}\label{PenaltyNonNegativeGamma}\nonumber
F_{\lambda_{n}}(u_{1}^{n},\ldots,u_{N}^{n})
\ge
\frac12
\int_{\Omega}
\sum_{i=1}^{N}|\nabla u_{i}^{n}|^{2}\,dx.
\end{equation}
Therefore,
\begin{equation}\label{LowerBoundGammaProof}\nonumber
F_{\infty}(\tilde{u}_{1},\ldots,\tilde{u}_{N})
\le
\liminf_{n\to\infty}
F_{\lambda_{n}}(u_{1}^{n},\ldots,u_{N}^{n}).
\end{equation}

\bigskip

We now prove the upper bound. Let $(\tilde{u}_{1},\ldots,\tilde{u}_{N})\in\mathcal H_{f}.$ Define
\begin{equation}\label{RecoverySequence}\nonumber
u_{i}^{n}:=\tilde{u}_{i},
\qquad
i=1,\ldots,N.
\end{equation}
Then
\begin{equation}\label{RecoveryWeakConvProof}\nonumber
(u_{1}^{n},\ldots,u_{N}^{n})
\rightharpoonup
(\tilde{u}_{1},\ldots,\tilde{u}_{N})
\quad\text{weakly in }\mathcal H.
\end{equation}
Since $(\tilde{u}_{1},\ldots,\tilde{u}_{N})\in\mathcal H_{f}$, we get
\begin{equation}\label{RecoveryEnergyEquality}\nonumber
F_{\lambda_{n}}(u_{1}^{n},\ldots,u_{N}^{n})
=
F_{\infty}(\tilde{u}_{1},\ldots,\tilde{u}_{N}).
\end{equation}
Therefore,
\begin{equation}\label{UpperBoundGammaProof}\nonumber
\limsup_{n\to\infty}
F_{\lambda_{n}}(u_{1}^{n},\ldots,u_{N}^{n})
=
F_{\infty}(\tilde{u}_{1},\ldots,\tilde{u}_{N}).
\end{equation}

\noindent Having established the $\Gamma$-convergence of $(F_{\lambda_n})$ to $F_\infty$, it remains to prove the convergence of the corresponding minimum values.

\noindent In fact, by Lemma~\ref{lem:constraint-limit}, every sequence of minimizers with uniformly bounded energy is relatively compact in the weak topology of \(\mathcal H\). Hence, the family \((F_{\lambda_{n}})\) is equicoercive in this topology. The fundamental theorem of \(\Gamma\)-convergence yields \eqref{MinimaConvergenceGamma}. This completes the proof.
\end{proof}

\bigskip

The third assertion of Theorem~\ref{thm_singularity} is established by the following proposition, which provides an explicit decomposition of the limiting energy.

\begin{proposition}\label{prop:Einfty-identification}
Assume that \(N\ge3\). Then the limiting energy \(E_{\infty}(f) \) satisfies
\begin{equation}\label{EinftyRepresentation}
E_{\infty}(f)
=
\frac{1}{2N}
\int_{\Omega}|\nabla f|^{2}\,dx
+
\min_{(w_{1},\ldots,w_{N})\in W_{f}}
\frac12
\sum_{i=1}^{N}
\int_{\Omega}|\nabla w_{i}|^{2}\,dx,
\end{equation}
where $W_{f}$ is defined by \eqref{Wf}.

\end{proposition}

\begin{proof}
Let $(u_1,\ldots,u_N)\in\mathcal H_f$ be a minimizer of $E_\infty(f)$, and define
\[
w_i=u_i-\frac1Nf,\qquad i=1,\ldots,N.
\]
Then $(w_1,\ldots,w_N)\in W_f$ and
\[
\nabla u_i=\frac1N\nabla f+\nabla w_i.
\]
Since $\sum_{i=1}^N w_i=0$, we have $\sum_{i=1}^N\nabla w_i=0$, and therefore
\[
\sum_{i=1}^N|\nabla u_i|^2
=\frac1N|\nabla f|^2+\sum_{i=1}^N|\nabla w_i|^2.
\]
Integrating over $\Omega$ gives
\begin{equation}\label{ineq}
\frac12\sum_{i=1}^N\int_\Omega |\nabla u_i|^2\,dx
=
\frac1{2N}\int_\Omega |\nabla f|^2\,dx
+\frac12\sum_{i=1}^N\int_\Omega |\nabla w_i|^2\,dx.
\end{equation}
Using the fact that $(u_1,\ldots,u_N)$ is a minimizer of $E_\infty(f)$, we obtain
\[
E_\infty(f)\ge
\frac1{2N}\int_\Omega |\nabla f|^2\,dx
+\inf_{W_f}
\frac12\sum_{i=1}^N\int_\Omega |\nabla w_i|^2\,dx.
\]

\noindent Conversely, let $(\widetilde w_1,\ldots,\widetilde w_N)\in W_f$ and define
\[
\widetilde u_i=\frac1Nf+\widetilde w_i,
\qquad i=1,\ldots,N.
\]
By the definition of $W_f$, we have
$(\widetilde u_1,\ldots,\widetilde u_N)\in\mathcal H_f$.
Hence,
\[
E_\infty(f)
\le
\frac12
\sum_{i=1}^N
\int_\Omega |\nabla\widetilde u_i|^2\,dx.
\]
By the same argument used to derive \eqref{ineq}, we obtain
\[
\frac12
\sum_{i=1}^N
\int_\Omega |\nabla\widetilde u_i|^2\,dx
=
\frac1{2N}
\int_\Omega |\nabla f|^2\,dx
+
\frac12
\sum_{i=1}^N
\int_\Omega |\nabla\widetilde w_i|^2\,dx.
\]

\noindent Taking the infimum over $W_f$ gives the reverse inequality, and therefore

\[
E_\infty(f)
=
\frac1{2N}\int_\Omega |\nabla f|^2\,dx
+
\inf_{(w_1,\ldots,w_N)\in W_f}
\frac12
\sum_{i=1}^N
\int_\Omega |\nabla w_i|^2\,dx.
\]
\medskip

To complete the proof, it remains to show that the infimum over $W_f$ is attained. To this end,
since $W_f \neq \varnothing$, let $(w_1^n, \ldots, w_N^n)_{n \geq 1} \subset W_f$
be a minimizing sequence for
\[
    m_f
    \;:=\;
    \inf_{(w_1,\ldots,w_N)\in W_f}
    \frac{1}{2}\sum_{i=1}^{N}
    \int_{\Omega}|\nabla w_i|^2\,dx.
\]
In particular, $\bigl(\sum_{i=1}^{N}\int_{\Omega}|\nabla w_i^n|^2\,dx\bigr)_{n}$
is bounded in $\mathbb{R}$.
Moreover, since $\bigl|\tfrac{1}{N}f + w_i^n\bigr| = 1$ and $|f| = 1$
almost everywhere in $\Omega$, we obtain
\[
    |w_i^n|
    \;\leq\;
    \Bigl|\tfrac{1}{N}f + w_i^n\Bigr| + \tfrac{1}{N}|f|
    \;=\;
    1 + \tfrac{1}{N}
    \qquad\text{a.e.\ in }\Omega.
\]
Therefore, $(w_1^n,\ldots,w_N^n)_{n}$ is bounded in
$\bigl(H^1(\Omega,\mathbb{R}^3)\bigr)^N$.
Hence, there exist $w_1,\ldots,w_N\in H^1(\Omega,\mathbb{R}^3)$ such that,
up to a subsequence and for every $i=1,\ldots,N$,
\[
    w_i^n \rightharpoonup w_i
    \quad\text{weakly in }H^1(\Omega,\mathbb{R}^3),
    \qquad
    w_i^n \to w_i
    \quad\text{strongly in }L^2(\Omega,\mathbb{R}^3).
\]
The strong $L^2$-convergence and the constraint
$\sum_{i=1}^{N}w_i^n = 0$ a.e.\ in $\Omega$ yield,
\[
    \sum_{i=1}^{N}w_i = 0
    \qquad\text{a.e.\ in }\Omega.
\]
Furthermore, $w_i^n(x)\to w_i(x)$ for a.e.\ $x\in\Omega$,
so passing to the pointwise limit in the constraint
$\bigl|\tfrac{1}{N}f(x) + w_i^n(x)\bigr| = 1$ gives
\[
    \Bigl|\tfrac{1}{N}f(x) + w_i(x)\Bigr| = 1
    \qquad\text{for a.e.\ }x\in\Omega.
\]
Thus $(w_1,\ldots,w_N)\in W_f$.
By the weak lower semicontinuity of the Dirichlet energy,
\[
    \frac{1}{2}\sum_{i=1}^{N}
    \int_{\Omega}|\nabla w_i|^2\,dx
    \;\leq\;
    \liminf_{n\to\infty}
    \frac{1}{2}\sum_{i=1}^{N}
    \int_{\Omega}|\nabla w_i^n|^2\,dx
    \;=\; m_f,
\]
while the reverse inequality holds by the definition of $m_f$.
Hence, the infimum $m_f$ is attained at $(w_1,\ldots,w_N)$,
which completes the proof of~\eqref{EinftyRepresentation}.\\
\end{proof}

Now, we will prove the gap phenomenon and the occurrence of
singular minimizers \eqref{GapLargeLambda}.

\begin{proposition}\label{prop_singularity}
Assume that \(N\ge3\). Assume that $f$ cannot be decomposed as $f=\sum_{i=1}^N u_i$ where each map $u_i\in H^1(\Omega,S^2)$ is a strong limit of smooth maps in $H^1(\Omega,S^2)$. Then there exists $\lambda_0>0$, depending on $f$ and $N$, such that for all $\lambda \ge \lambda_0$, any minimizer $(u_1^\lambda,\dots,u_N^\lambda)$ of $F_\lambda$ is not regular in $\Omega$. In particular, at least one
component \(u_i^\lambda\) admits a singularity in \(\Omega\). Moreover, 
\begin{equation}\label{GapLargeLambda_hcn}
E_{\lambda}(f)<E_{\lambda,\mathrm{reg}}(f).
\end{equation}
\end{proposition}

\begin{proof}
By Proposition~\ref{lem:uniform-bound},
\begin{equation}\label{energybound}
E_{\lambda}(f)<+\infty,
\qquad
\forall \lambda>0.
\end{equation}

\noindent Let \((\lambda_{n})\) and \((\varepsilon_{n})\) be two sequences of positive numbers such that $ \lambda_{n}\to+\infty$ and $\varepsilon_{n}\to 0$ as $n$ tends to  $+\infty$. For every \(n\in\mathbb N\), there exists $
(u_{1}^{n},\ldots,u_{N}^{n})
\in C^{1}(\overline{\Omega},S^{2})^{N}
$
such that
\begin{equation}\label{RegularAlmostMinimizer}
F_{\lambda_{n}}(u_{1}^{n},\ldots,u_{N}^{n})
\le
E_{\lambda_{n},\mathrm{reg}}(f)+\varepsilon_{n}.
\end{equation}
We may assume that \(E_{\lambda_{n},\mathrm{reg}}(f)\) is bounded. Otherwise, by \eqref{energybound}, the conclusion is satisfied. Thus,
$F_{\lambda_{n}}(u_{1}^{n},\ldots,u_{N}^{n})\le C.$

\noindent Applying Lemma~\ref{lem:constraint-limit}, we deduce that, up to a subsequence still denoted by $(u_{1}^{n},\ldots,u_{N}^{n}),$ there exist maps $ \tilde{u}_{1},\ldots,\tilde{u}_{N}\in H^{1}(\Omega,S^{2})$ such that for every \(i=1,\ldots,N\),
\begin{equation}\label{WeakConvergenceProof}
u_{i}^{n}\rightharpoonup \tilde{u}_{i}
\quad\text{weakly in }H^{1}(\Omega,S^{2}),
\end{equation}
and
$$\sum_{i=1}^{N}\tilde{u}_{i}=f
\quad\text{a.e. in }\Omega.$$
Since \(f\) cannot be written as a sum of \(N\) maps that are strong \(H^1(\Omega,S^2)\)-limits of smooth sphere-valued maps, the above convergences \eqref{WeakConvergenceProof} cannot all be strong in \(H^1(\Omega,S^2)\).
Hence, there exists \(k\in\{1,\ldots,N\}\) such that
$$u_{k}^{n}\nrightarrow \tilde{u}_{k}
\quad\text{strongly in }H^{1}(\Omega,S^{2}).$$
Consequently,
$$\int_{\Omega}|\nabla \tilde{u}_{k}|^{2}\,dx
<
\liminf_{n\to+\infty}
\int_{\Omega}|\nabla u_{k}^{n}|^{2}\,dx.$$
By weak lower semicontinuity, for every \(i\neq k\),
$$\int_{\Omega}|\nabla \tilde{u}_{i}|^{2}\,dx
\le
\liminf_{n\to+\infty}
\int_{\Omega}|\nabla u_{i}^{n}|^{2}\,dx.$$
Therefore,
$$\frac12
\sum_{i=1}^{N}
\int_{\Omega}|\nabla \tilde{u}_{i}|^{2}\,dx
<
\alpha
:=
\liminf_{n\to+\infty}
\frac12
\sum_{i=1}^{N}
\int_{\Omega}|\nabla u_{i}^{n}|^{2}\,dx.$$

\noindent Then, for every \(\varepsilon>0\), there exists \(N_{0}\in\mathbb N\) such that, for every \(n\ge N_{0}\),
$$E_{\lambda_{n}}(f)
\le
F_{\lambda_{n}}(\tilde{u}_{1},\ldots,\tilde{u}_{N})
=
\frac12
\sum_{i=1}^{N}
\int_{\Omega}|\nabla \tilde{u}_{i}|^{2}\,dx
<
\alpha-\varepsilon,$$
and
$$\alpha-\varepsilon
\le
\frac12
\sum_{i=1}^{N}
\int_{\Omega}|\nabla u_{i}^{n}|^{2}\,dx.$$
Hence
$$E_{\lambda_{n}}(f)
<
\frac12
\sum_{i=1}^{N}
\int_{\Omega}|\nabla u_{i}^{n}|^{2}\,dx
\le
F_{\lambda_{n}}(u_{1}^{n},\ldots,u_{N}^{n}).$$
Therefore, by \eqref{RegularAlmostMinimizer},
\begin{equation}\label{GapApproxReg}
E_{\lambda_{n}}(f)
<
F_{\lambda_{n}}(u_{1}^{n},\ldots,u_{N}^{n})
\le
E_{\lambda_{n},\mathrm{reg}}(f)+\varepsilon_{n}.
\end{equation}
Thus, \eqref{GapApproxReg} implies that, for \(n\) sufficiently large,
\[
E_{\lambda_n}(f)
<
E_{\lambda_n,\mathrm{reg}}(f).
\]

\noindent Therefore, there exists \(\lambda_0>0\) depending on $f$ and $N$  such that
\[
E_{\lambda}(f)
<
E_{\lambda,\mathrm{reg}}(f)
\qquad
\forall \lambda\ge \lambda_0.
\]

\noindent Consequently, no minimizer of \(F_\lambda\) can be regular for \(\lambda\ge\lambda_0\). Hence, every minimizer has at least one
singular point in \(\Omega\).
\end{proof}
\noindent Combining the above propositions, the proof of Theorem \ref{thm_singularity} is now complete.

\vspace{0.5cm}

\noindent\textbf{Authors contributions.}
The authors are listed alphabetically by their surnames, and all authors have made equal contributions.

\medskip

\noindent\textbf{Conflict of interest.}
The authors declare that they have no conflict of interest.

\medskip

\noindent\textbf{Data availability.}
 Data were not used in this study.


\begin{thebibliography}{99}

\bibitem{B87}
H. Brezis,
\emph{Liquid crystals and energy estimates for $S^{2}$-valued maps},
in \emph{Theory and Applications of Liquid Crystals} (J. Ericksen and D. Kinderlehrer, eds.),
IMA Vol. Math. Appl. \textbf{5}, Springer, New York, 1987, pp. 31--52.

\bibitem{GP}
E. G. Virga,
\emph{Variational Theories for Liquid Crystals},
Chapman \& Hall, London, 1994.

\bibitem{MAJ}
P. G. de Gennes, J. Prost,
\emph{The Physics of Liquid Crystals},
2nd ed., Oxford University Press, 1993.

\bibitem{A}
F. Alouges,
\emph{A new algorithm for computing liquid crystal stable configurations: the harmonic mapping case},
SIAM J. Numer. Anal. \textbf{34} (1997), 1708--1726.

\bibitem{BG}
A. DeSimone, R. V. Kohn, S. M\"uller, F. Otto,
\emph{A reduced theory for thin-film micromagnetics},
Comm. Pure Appl. Math. \textbf{55} (2002), 1408--1460.




\bibitem{LC1}
J. M. Ball,
\emph{Liquid crystals and their defects},
in \emph{Mathematical Thermodynamics of Complex Fluids}
(CIME Foundation Subseries),
Springer, Cham, 2017, pp. 1--46.

\bibitem{LC2}
A. Majumdar, A. Zarnescu,
\emph{Landau--De Gennes theory of nematic liquid crystals:
the Oseen--Frank limit and beyond},
Arch. Ration. Mech. Anal. \textbf{196} (2010), no. 1, 227--280.


\bibitem{LC3}
W. Wang, L. Zhang, P. Zhang,
\emph{Modelling and computation of liquid crystals},
Acta Numer. \textbf{30} (2021), 765--851.


\bibitem{MG1}
G. Di Fratta, C. B. Muratov, F. N. Rybakov, V. V. Slastikov,
\emph{Variational principles of micromagnetics revisited},
SIAM J. Math. Anal. \textbf{52} (2020), no. 4, 3580--3601.

\bibitem{MG2}
E. Davoli, G. Di Fratta,
\emph{Homogenization of chiral magnetic materials:
a mathematical evidence of Dzyaloshinskii's predictions on helical structures},
J. Nonlinear Sci. \textbf{30} (2020), 1229--1262.

\bibitem{MG3}
E. Davoli, G. Di Fratta, D. Praetorius, M. Ruggeri,
\emph{Micromagnetics of thin films in the presence of
Dzyaloshinskii--Moriya interaction},
Math. Models Methods Appl. Sci. \textbf{32} (2022), no. 5, 911--939.



\bibitem{IM1}
K. Bredies, M. Holler, M. Storath, A. Weinmann,
\emph{Total generalized variation for manifold-valued data},
SIAM J. Imaging Sci. \textbf{11} (2018), no. 3, 1785--1848.


\bibitem{MG4}
G. Di Fratta, V. Slastikov, A. Zarnescu,
\emph{On a sharp Poincar\'e-type inequality on the 2-sphere
and its application in micromagnetics},
SIAM J. Math. Anal. \textbf{51} (2019), no. 6, 4684--4713.


\bibitem{SU1}
R. Schoen,  K. Uhlenbeck,
\emph{A regularity theory for harmonic maps},
J. Differential Geom. \textbf{17} (1982), 307--335.

\bibitem{SU2}
R. Schoen, K. Uhlenbeck,
\emph{Boundary regularity and the Dirichlet problem for harmonic maps},
J. Differential Geom. \textbf{18} (1983), 253--268.

\bibitem{HKL}
R. Hardt, D. Kinderlehrer, F.-H. Lin,
\emph{Existence and partial regularity of static liquid crystal configurations},
Comm. Math. Phys. \textbf{105} (1986), 547--570.

\bibitem{BBC}
F. Bethuel, H. Brezis, J.-M. Coron,
\emph{Minimisation de $\int|\nabla(u-x/|x|)|^{2}$ et divers phénomènes de gap},
C. R. Acad. Sci. Paris Sér. I Math. \textbf{310} (1990), 859--864.

\bibitem{BCL}
F. Bethuel, H. Brezis,  J.-M. Coron,
\emph{Relaxed energies for harmonic maps},
in \emph{Variational Problems} (H. Berestycki, J.-M. Coron, I. Ekeland, eds.),
Birkhäuser, Basel, 1990, pp. 37--52.

\bibitem{HL}
R. Hardt, F.-H. Lin,
\emph{A remark on $H^{1}$ mappings},
Manuscripta Math. \textbf{56} (1986), 1--10.

\bibitem{LM}
F. Lin, C. Wang,
\emph{The Analysis of Harmonic Maps and Their Heat Flows},
World Scientific, Singapore, 2008.

\bibitem{Z}
F. Zhou,
\emph{Recent developments in the regularity theory for sphere-valued harmonic maps},
in \emph{Geometric Analysis and PDEs} (L. Ambrosio et al., eds.),
Lecture Notes in Math. \textbf{2267}, Springer, 2020, pp. 145--169.



\bibitem{HZ}
R. Hadiji, F. Zhou,
\emph{Regularity of $\int_{\Omega}|\nabla u|^{2}+\lambda\int_{\Omega}|u-f|^{2}$ and some gap phenomenon},
Potential Anal. \textbf{1} (1992), 385--400.

\bibitem{HaZh}
P. Courilleau, S. Dumont, R. Hadiji,
\emph{Regularity of minimizing maps with values in $S^2$ and some numerical simulations},
Adv. Math. Sci. Appl. \textbf{10} (2020), no. 2, 711--733.


\bibitem{GH09}
A. Gaudiello, R. Hadiji,
\emph{Asymptotic analysis, in a thin multidomain, of minimizing maps with values in $S^{2}$},
Ann. Inst. H. Poincaré Anal. Non Linéaire \textbf{26} (2009), no.~1, 59--80.


























\end{thebibliography}
\end{document}